\documentclass[12pt,a4paper]{article}
\usepackage{amsmath,amsthm,amsfonts,amssymb,amscd, amsxtra,color}
\usepackage[active]{srcltx}
\usepackage{url}
\usepackage{cite}
\usepackage{multirow}
\newtheorem{theorem}{Theorem}

\newtheorem{proposition}{Proposition}
\newtheorem{remark}{Remark}
\newtheorem{definition}{Definition}

\def\min{\operatorname{Minimize}}

\begin{document}

\title{How to project onto extended second order cones}

\author{
 O. P. Ferreira\thanks{IME/UFG, Avenida Esperança, s/n, Campus Samambaia,  Goi\^ania, GO, 74690-900, Brazil (e-mail:{\tt
orizon@ufg.br}).  The author was supported in part by FAPEG, CNPq Grant  305158/2014-7 and PRONEX--Optimization(FAPERJ/CNPq).}  
\and
S. Z. N\'emeth \thanks{School of Mathematics, University of Birmingham, Watson Building, Edgbaston, Birmingham B15 2TT, United Kingdom
(e-mail:{\tt s.nemeth@bham.ac.uk}).} 
}

\maketitle

\begin{abstract}
 The extended second order cones were introduced by S. Z. N\'emeth and G. Zhang in [S. Z. N\'emeth and G. Zhang. {\it Extended Lorentz cones and variational inequalities on cylinders}. J. Optim. Theory Appl., 168(3):756-768, 2016]  for solving mixed complementarity problems and variational inequalities on cylinders. R. Sznajder in [R. Sznajder. {\it The Lyapunov rank of extended second order cones}. Journal of Global Optimization, 66(3):585-593, 2016] determined the automorphism groups and the Lyapunov or bilinearity ranks of these cones. S. Z. N\'emeth and G. Zhang  in [S.Z. N\'emeth and G. Zhang. {\it Positive operators of Extended Lorentz cones}. arXiv:1608.07455v2,2016]  found both necessary conditions and sufficient conditions for a linear operator to be a positive operator of an extended second order cone. In this  note we give formulas for projecting onto the extended second order cones. In the most general case the formula depends  on a  piecewise linear equation for one real variable which  is  solved by using numerical methods.
 
\end{abstract}

\section{Introduction}

The Lorentz cone is an important object in theoretical physics. In recent times it has been rebranded as second order cone and used for various 
application
in optimization. Some robust optimization, plant location and investment portfolio manangement problems were formulated as as a second order cone
program \cite{AG2003}. Another good survey paper with a wide range of applications of second order cone programming is \cite{MR1655138}. More
recent connections of second order cone programming and second order cone complementarity problem with physics, mechanics, economics, game theory,
robotics, optimization and neural networks were considered in
\cite{MR3158056,MR2377196,MR2116450,MR3010551,MR2925039,MR2568432,MR2522815,KCY2011,MR2179239}. The importance of the second order cone   is nowadays notorious not only in theoretical physics, but in optimization as well. 

Thus far, there is no closed-form expression for metric (orthogonal) projection onto a general closed convex cone.  A nice property of the second order cone is that it admits an explicit representation of the projection mapping onto it (see 
\cite[Proposition 3.3]{MR1885570}). The
original motivation for extending the second order cone was inspired by using iterative methods for solving complementarity problems and variational
inequalities \cite{NZ20151,NZ2016a}. These iterative methods are based on the property that the projection onto the closed convex set defining 
the problem is isotone with respect to the order defined out by a cone. Usually this is a very restrictive condition. However, cylinders and in
particular cylinders with cone base admit isotone projections onto them with respect to the extended second order cones. Therefore, variational
inequalities on cylinders and mixed complementarity problems can be solved by using such iterative techniques based on monotone convergence
\cite{NZ20151,NZ2016b}. 

Later it turned out that many of these cones could be even more useful because the bilinearity rank (or Lyapunov rank)
\cite{RudolfNoyanPappAlizadeh2011,GowdaTao2014,GowdaTrott2014,Trott2014,OrlitzkyGowda2016} of  them  is higher than the dimension of the
underlying space and therefore they have good numerical properties. More specifically, for $p>1$ this is true whenever $q^2-3q+2>2p$
\cite{RS2016},  where $p$, $q$ are from the definition of the extended second order cone (see Definition \ref{def-ext-lor}). Such cones are
``numerically good'' cones when solving complementarity problems defined on them. The extended second order cones are also irreducible
\cite{RS2016}. But to be really usable from optimization point of view we need easy ways of projecting onto them.  In this  paper we show that
projecting onto an extended second order cone it is ``almost possible''  by using closed-form  expressions. We present a set of formulas for
projecting onto an extended second order cone which is  subject to solving a piecewise linear equation with one real variable only. The method of
finding these expressions is based on the special form of the complementarity set of the extended second order cone and Moreau's decomposition
theorem \cite{MR0139919} for projecting onto cones. The latter problem of projecting onto the extended second order cone is a particular conic
optimization problem with respect to this cone. Although, the problem of projecting the point $(x,u)\in\mathbb R^p\times\mathbb R^q$ into the
extended second order cone $L$ (see Definition \ref{def-ext-lor}) can be transformed into the second order conic optimization problem
\[
\min \left\{ \|y-x\|^2+\|v-u\|^2~:~ (y,v)\in\mathbb R^p\times\mathbb R^q, ~ ~ \ell_i(y,v)\in\mathcal{L},\quad i=1,\ldots,p\right\}, 
\]
where $\mathcal{L}=\{(t,u)\in\mathbb R\times\mathbb R^q:t\ge\|u\|\}$ is the second order cone in $\mathbb R^{q+1}\equiv\mathbb R\times\mathbb R^q$
and $\ell_i:\mathbb R^p\times\mathbb R^q\to\mathbb R\times\mathbb R^q$ are the linear mappings defined by $\ell_i(y,v)=(y_i,v)$, the
complexity of our method is much simpler than solving the reformulated problem, because apart from closed-form expressions, it contains only one
piecewise linear equation. By considering such a reformulation one would lose the useful special structure of the cone, which is the cornerstone
for the simplicity of our method.

Certainly, the explicit representation of the projection mapping onto the second order cone (see \mbox{
\cite[Proposition 3.3]{MR1885570}}\hspace{0pt}) should not be 
handled as a conic optimization problem and the need to solve a simple piecewise linear equation for $p>1$ makes our method just slightly more complex. The above observation about why one shouldn't  reformulate the projection onto the extended second order cone into a second order conic
optimization problem, together with the irreducibility of the second order cone, clearly shows that this cone ``deserves a closer look''.

The structure of the paper is as follows: In Section 2 we fix the notation and the terminology used throughout the paper. In Section 3 we present the formulas for projecting onto the
extended second order cone. In Section 4 we solve the piecewise linear equation involved in these formulas by using the semi-smooth Newton's  method and  a method based on Picard's iteration. Finally, we make some remarks in the last section. 
\section{Preliminaries}
Let $\ell,m,p,q$ be positive integers such that $m=p+q$. We identify the the vectors of ${\mathbb R}^\ell$ with $\ell\times 1$ matrices with real 
entries. The scalar product in ${\mathbb R}^\ell$ is defined by the mapping \[{\mathbb R}^\ell\times{\mathbb R}^\ell\ni (x,y)\mapsto\langle x,y \rangle:=x^\top y\in{\mathbb R}\] and the
corresponding norm by
\[{\mathbb R}^\ell\ni x\mapsto\|x\|:=\sqrt{\langle x,x\rangle}\in{\mathbb R}.\] For $x,y\in{\mathbb R}^\ell$ denote $x\perp y$ if $\langle x,y\rangle=0$.
We identify the elements of
${\mathbb R}^p\times{\mathbb R}^q$ with the elements of ${\mathbb R}^m$ through the correspondence \[{\mathbb R}^p\times{\mathbb R}^q\ni(x,y)\mapsto (x^\top,y^\top)^\top.\] Through this
identification the scalar product in ${\mathbb R}^p\times{\mathbb R}^q$ is defined by 
\[\langle (x,y),(u,v)\rangle:=\langle (x^\top,y^\top)^\top,(u^\top,v^\top)^\top\rangle=\langle x,u\rangle+\langle y,v\rangle.\] A closed set $K\subset{\mathbb R}^\ell$ with nonempty interior is
called a \emph{proper cone} if $K+K\subset K$, $K\cap(-K)=\{0\}$ and $\lambda K\subset K$, for any $\lambda$ positive real number. The \emph{dual
cone} of a proper cone $K\subset{\mathbb R}^\ell$ is a proper cone defined by \[K^*:=\{x\in{\mathbb R}^\ell~:~\langle x,y\rangle\ge0,\mbox{ }\forall y\in K\}.\] A proper cone 
$K\subset{\mathbb R}^\ell$ is called subdual if $K\subset K^*$, superdual if $K^*\subset K$ and self-dual if $K^*=K$. If $K,D\subset{\mathbb R}^\ell$ are proper cones
such that $D=K^*$, then $D^*=K$ and the cones $K$, $D$ are called \emph{mutually dual}.

For a proper cone $K\in{\mathbb R}^\ell$ denote 
\[C(K):=\left\{(x,y)\in K\times K^*~:~x\perp y\right\}\] the \emph{complementarity set} of $K$.

Let $C\in{\mathbb R}^\ell$ be a closed convex set. The projection mapping $P_C\colon{\mathbb R}^\ell\to{\mathbb R}^\ell$ onto $C$ is the mapping defined by 
\[P_C(x):=\mbox{argmin}\{\|x-y\|:y\in C\}.\] 
We recall here Moreau's decomposition Theorem \cite{MR0139919} (stated here for proper cones only):
\begin{theorem} \label{th:mt}
	Let $K\subset{\mathbb R}^\ell$ be a proper cone, $K^*$ its dual cone and $z\in{\mathbb R}^\ell$. Then, the following two statements are equivalent:
	\begin{enumerate}
		\item[(i)] $z=x-y$ and $(x,y)\in C(K)$,
		\item[(ii)] $x=P_K(z)$ and $y=P_{K^*}(-z)$.
	\end{enumerate}
\end{theorem}
Theorem~\ref{th:mt} implies  \[z=P_K(z)-P_{K^*}(-z),\] with $P_K(z)\perp P_{K^*}(-z)$.

For  $z\in{\mathbb R}^\ell$ we denote  $z=(z_1,\dots,z_\ell)^\top$. Let $\ge$ denote the component-wise order in 
${\mathbb R}^\ell$, that is, the order defined by \({\mathbb R}^\ell\ni x\ge y\in{\mathbb R}^\ell\) if and only if $x_i\ge y_i$ for $i=1,\dots,\ell$.  Denote by $0$ the vector in ${\mathbb R}^\ell$ or a scalar zero (it will not lead to any confusion), by $e$ the vector  of ones in ${\mathbb R}^\ell$  and by ${\mathbb R}^\ell_+=\{x\in{\mathbb R}^\ell~:~ x\ge0\}$ the nonnegative orthant. The proper cone ${\mathbb R}^\ell_+$ is self-dual. For a real number $\alpha\in{\mathbb R}$ denote $\alpha^+:=\max(\alpha,0)$ and $\alpha^-:=\max(-\alpha,0)$. 
For a vector $z\in{\mathbb R}^\ell$ denote $z^+:=(z_1^+,\dots,z_\ell^+)$, $z^-:=(z_1^-,\dots,z_\ell^-)$, $|z|:=(|z_1|,\dots,|z_\ell|)$,
$\mbox{sgn}(z):=(\mbox{sgn}(z_1),\dots,\mbox{sgn}(z_\ell))$
and $\mbox{diag}(z)$ the $\ell\times\ell$ diagonal matrix with entries $\mbox{diag}(z)_{ij}:=\delta_{ij}z_i$, where $i,j\in\{1,\dots,\ell\}$. It is known that $z^+=P_{{\mathbb R}^\ell_+}(z)$ and  $z^-=P_{{\mathbb R}^\ell_+}(-z)$. 

We recall from \cite{NZ20151} the following definition of a pair of mutually dual extended second order cones $L$, $M$:
\begin{definition}\label{def-ext-lor}
\begin{align*}
L&:=\left\{(x,u)\in{\mathbb R}^p\times{\mathbb R}^q:x\ge\|u\|e\right\}, \\
M&:=\left\{(x,u)\in{\mathbb R}^p\times{\mathbb R}^q:\langle x,e\rangle\ge\|u\|,x\ge0\right\}.
\end{align*}
where $\ge$ denotes the component-wise order. 
\end{definition}

It is known that both $L$ and $M$ are
proper cones, $L$ is subdual $M$ is superdual and if $p=1$, then both cones reduce to the second order cone. The cones $L$ and $M$ are polyhedral if
and only if $q=1$. If we allow $q=0$ as well, then the cones $L$ and $M$ reduce to the nonnegative orthant. More properties of the extended
second order cones can be found in \cite{NZ20151,RS2016,NZ2016b}.
\section{Projection formulas for extended second order cones}
In this section we give formulas for projecting onto the pair of mutually dual extended second order cones.  Before presenting our main  theorem, we need some preliminary results for these cones. Let $p,q$  be positive integers.
\begin{proposition}\label{pm}
	Let $x,y\in{\mathbb R}^p$ and $u,v\in{\mathbb R}^q\setminus\{0\}$. We have that $(x,u,y,v):=((x,u),(y,v))\in C(L)$ if and only if there exists a $\lambda>0$
	such that $v=-\lambda u$, $\langle y,e\rangle=\|v\|$ and $(x-\|u\|e,y)\in C({\mathbb R}^p_+)$.
\end{proposition}
\begin{proof}
	Suppose first that there exists $\lambda>0$ such that $v=-\lambda u$, $\langle y,e\rangle=\|v\|$ and $(x-\|u\|e,y)\in C({\mathbb R}^p_+)$. Hence,
	$(x,u)\in L$ and $(y,v)\in M$. Moreover, 
	\[\langle (x,u),(y,v)\rangle=\langle x,y\rangle+\langle u,v\rangle=\|u\|\langle e,y\rangle-\lambda\|u\|^2=\|u\|\|v\|-\lambda\|u\|^2=0.\] Thus, $(x,u,y,v)\in
	C(L)$. Conversely, suppose that $(x,u,y,v)\in C(L)$. Then, $(x,u)\in L$, $(y,v)\in M$ and
	\[0=\langle (x,u),(y,v)\rangle=\langle x,y\rangle+\langle u,v\rangle\ge\langle\|u\|e,y\rangle+\langle u,v\rangle\ge\|u\|\|v\|+\langle u,v\rangle\ge0.\]
	Hence, there exists $\lambda>0$ such that $v=-\lambda u$, $\langle e,y\rangle=\|v\|$ and $\langle x-\|u\|e,y\rangle=0$. It follows that
	$(x-\|u\|e,y)\in C({\mathbb R}^p_+)$. \qed
\end{proof}

Before presenting the main result of this section we introduce a piecewise linear function and establish  some important properties of it. This
function will play an important role in the sequel, namely,  the   formulas for the projection will depend on its single positive zero. The piecewise linear function $ \psi\colon[0, +\infty) \to {\mathbb R}$ is defined by 
\begin{equation}\label{eq:psy}
				\psi(\lambda):=-\lambda \|w\|+\left\langle e,[(\lambda+1)z-\|w\|e]^-\right\rangle.
\end{equation}
For stating   the next proposition we need to define the following diagonal matrix, which we will see  is related to   the  subdifferential $\partial \psi$  of $\psi$:
\begin{equation}\label{def:N(x)}
				N(\lambda):=\mbox{diag}\left(-\mbox{sgn}\left( [(\lambda+1)z-\|w\|e]^-\right)\right),  \qquad \lambda\in  [0, +\infty).
			\end{equation}
\begin{proposition}\label{pr:psy}
The function  $\psi$ is convex. Moreover,  if 
$$
z^+\not\ge\|w\|e, \qquad \langle z^-,e\rangle<\|w\|, 
$$
then we have:
 \begin{enumerate}
  \item $-\|w\|+\left\langle e, N(\lambda)z\right\rangle \in \partial \psi(\lambda)$ and  $-\|w\|+\left\langle e, N(\lambda)z\right\rangle <0$,  for all
	  $\lambda \geq 0$;
 \item $\psi$ has a unique zero $\lambda_* >0$.
 \end{enumerate}
\end{proposition}
\begin{proof}
We first note that  the function   $\psi$  can be  equivalently  given by 
\begin{equation} \label{def:edpsi}
 \psi(\lambda):=-\lambda \|w\|+\sum_{i=1}^{p} \psi_i(\lambda), \qquad    \psi_i(\lambda):= [(\lambda+1)z_i-\|w\|]^-,  \qquad \lambda \geq 0.
\end{equation}
 Since the  sum and  the maximum of two convex functions is convex, it follows that the function 
 $\psi_i(\lambda)=\max\{-(\lambda+1)z_i+\|w\|,~0\}$ is convex for all  $i=1, \ldots p$. Hence, the  result of the first part follows. 
	\begin{enumerate}
\item  The definitions  of   $ \psi$ and   $\psi_i$ in  \eqref{def:edpsi} imply that  $\partial \psi(\lambda)= -\|w\| + \sum_{i=1}^{p} \partial \psi_i(\lambda)$. Moreover,   considering that   $ \psi_i(\lambda)=\max\{-(\lambda+1)z_i+\|w\|,~0\}$, we have    $ -\mbox{sgn}\left(
	[(\lambda+1)z_i-\|w\|]^-\right)z_i\in \partial \psi_i(\lambda)$,   for all  $i=1, \ldots p$. Therefore,  using \eqref{def:N(x)},  the inclusion follows.  To prove the inequality, note that   \eqref{def:N(x)} implies that the entries of  $N(\lambda)$ are equal to $0$ or
	$-1$, for all  $\lambda \geq 0$. Thus,  from the assumption $\langle z^-,e\rangle<\|w\|$ we have $-\|w\| +\left\langle e, N(\lambda) z\right\rangle <
	0$, for all  $\lambda \geq 0$.

\item First, we  show that \eqref{eq:psy} has a positive zero. Note that $z\not\geq \|w\|e$, otherwise it would follow that $z^+=z\geq \|w\|e$, which
	contradicts our assumptions. Then, there exists $ i_0\in \{1, \ldots, p\}$ such that $z_{i_0} <\|w\|$. Hence, from
	\eqref{def:edpsi} we have $\psi(0)> \|w\|-  z_{i_0}> 0$.   If $\lambda>0$ is sufficiently
			large, then $\mbox{sgn}[(\lambda+1)z_i-\|w\|]= \mbox{sgn} z_i$ and consequently 
			$[(\lambda+1)z_i-\|w\|]^-\leq (\lambda+1)z_i^-+\|w\|$. By using the  last  inequality, \eqref{def:edpsi} and the
			assumption $\langle z^-,e\rangle<\|w\|$,  we conclude that   for $\lambda>0$  sufficiently  large,   it is true that
			 \begin{multline}
			 \psi(\lambda)\leq -\lambda\|w\|+\left\langle e,(\lambda+1)z^-+\|w\|e\right\rangle= \\\left[-\|w\|+\left\langle
			z^-,e\right\rangle\right]\lambda+ \|w\| +\left\langle e,z^-\right\rangle<0.
			\end{multline}
			Since $\psi$ is continuous, there is a $\lambda_*>0$ such that $\psi(\lambda_*)=0$.  By contradiction we assume that $\psi$  has two positive zeroes  $\bar{\lambda}$ and  
			$\hat{\lambda}$.  Let $0<\hat{\lambda}< \bar{\lambda}$. Since $\psi$ is convex and 
			$-\|w\|+ \left\langle e, N(\lambda)z\right\rangle \in \partial \psi(\lambda)$, 
			we have $ \psi(\hat{\lambda})\geq   \psi(\bar{\lambda}) +[-\|w\|+ \left\langle e,
			N(\bar{\lambda})z\right\rangle][\hat{\lambda}-\bar{\lambda}] $. Due to  $\psi(\hat{\lambda})=\psi(\bar{\lambda})=0$ and
			considering that $0<\hat{\lambda}<\bar{\lambda}$, the last inequity implies that $- \|w\|+ \left\langle e,
			N(\lambda)z\right\rangle\geq 0$,  which contradicts the second part of item 1. Therefore, $\psi$ has a unique positive zero. \qed
	\end{enumerate} 
\end{proof}
Now we ready to state and prove the main result of the paper.
\begin{theorem}\label{th:pelc}
	Let $(z,w)\in{\mathbb R}^p\times{\mathbb R}^q$. Then, we have
	\begin{enumerate}
		\item If $z^+\ge\|w\|e$, 
			then $P_L(z,w)=(z^+,w)$ and $P_M(-z,-w)=(z^-,0)$.
		\item If $\langle z^-,e\rangle\ge\|w\|$, 
			then $P_L(z,w)=(z^+,0)$ and $P_M(-z,-w)=(z^-,-w)$.
		\item If $z^+\not\ge\|w\|e$ and $\langle z^-,e\rangle<\|w\|$,
			then the piecewise linear equation
			\begin{equation}\label{ess}
				\lambda \|w\|=\left\langle e,[(\lambda+1)z-\|w\|e]^-\right\rangle.
			\end{equation}
			has a unique positive solution $\lambda>0$,
			\begin{equation}\label{epl}
				P_L(z,w)=\left(\left[z-\frac{1}{\lambda +1}\|w\|e\right]^++\frac{1}{\lambda+1}\|w\|e,~\frac{1}{\lambda+1}w\right)
			\end{equation}
			and
			\begin{equation}\label{epm}
				P_M(-z,-w)=\left(\left[z-\frac{1}{\lambda +1}\|w\|e\right]^-,~-\frac{\lambda}{\lambda+1}w\right)
			\end{equation}
	\end{enumerate}
\end{theorem}

\begin{proof}
	We will use Moreau's decomposition theorem for $L$ for proving all three items. In this case this theorem states that, $P_L(z,w)=(x,u)$ 
	and $P_M(-z,-w)=(y,v)$ if and only if $(z,w)=(x,u)-(y,v)$ and $(x,u,y,v)\in C(L)$. 
	\begin{enumerate}
		\item This is exactly the case when $v=0$. 

			Indeed, $v=0$ implies $P_L(z,w)=(x,u)$ and $P_M(-z,-w)=(y,0)$. Hence, $z=x-y$, $w=u$,
			$x\ge\|u\|e$, $y\ge0$ and $\langle x,y\rangle=0$. By using Moreau's decomposition theorem for ${\mathbb R}^p_+$, we have that $z=x-y$,
			$x\ge0$, $y\ge0$ and $\langle x,y\rangle=0$ implies $x=z^+$ and $y=z^-$. Since, $w=u$ and $x\ge\|u\|e$, we get
			$z^+\ge\|w\|e$.  

			Conversely, suppose that $z^+\ge\|w\|e$. Then  $(z^+,w,z^-,0)\in C(L)$. Hence, by Moreau's
			decomposition Theorem for $L$, we get $P_L(z,w)=(z^+,w)$ and $P_M(-z,-w)=(z^-,0)$. Thus, $v=0$. 
		\item This is exactly the case when $u=0$. 

			Indeed, $u=0$ implies $P_L(z,w)=(x,0)$ and $P_M(-z,-w)=(y,v)$. Hence, $z=x-y$, $w=-v$, $x\ge0$,
			$\langle y,e\rangle\ge\|v\|$, $y\ge0$ and $\langle x,y\rangle=0$. By using Moreau's decomposition theorem for ${\mathbb R}^p_+$, we have that $z=x-y$,
			$x\ge0$, $y\ge0$ and $\langle x,y\rangle=0$ implies $x=z^+$ and $y=z^-$. Since $w=-v$ and $\langle y,e\rangle\ge\|v\|$, we get
			$\langle z^-,e\rangle\ge\|w\|$. 

			Conversely, suppose that $\langle z^-,e\rangle\ge\|w\|$. Then, it is easy to check that $(z^+,0,z^-,-w)\in C(L)$. Then, by Moreau's
			decomposition Theorem for $L$, we get $P_L(z,w)=(z^+,0)$ and $P_M(-z,-w)=(z^-,-w)$. Thus, $u=0$. 
		\item This is exactly the case when $u\ne 0$ and $v\ne 0$. 

			From Proposition \ref{pm} it follows that $(z,w)=(x,u)-(y,v)$ and $(x,u,y,v)\in C(L)$ is equivalent to $z=x-y$, $w=u-v$ and
			the existence of a $\lambda>0$ such that $v=-\lambda u$, $\langle y,e\rangle=\|v\|$ and $(x-\|u\|e,y)\in C({\mathbb R}^p_+)$. On the
			other hand, by Moreau's decomposition theorem for ${\mathbb R}^p_+$, $(x-\|u\|e,y)\in C({\mathbb R}^p_+)$ is equivalent
			to $x-\|u\|e=[x-\|u\|e-y]^+$ and $y=[x-\|u\|e-y]^-$. Hence, 
			\begin{equation}\label{epli}
				P_L(z,w)=\left(x,\frac{1}{\lambda+1}w\right)
			\end{equation}
			and 
			\begin{equation}\label{epmi}
				P_M(-z,-w)=\left(y,-\frac{\lambda}{\lambda+1}w\right)
			\end{equation} 
			if and only if $z=x-y$ and $\lambda>0$ is such that 
			\begin{equation}\label{ebm}
				\langle y,e\rangle=\frac{\lambda}{\lambda+1}\|w\|,
			\end{equation} 
			\begin{equation}\label{ex}
				x=\left[z-\frac{1}{\lambda+1}\|w\|e\right]^++\frac{1}{1+\lambda}\|w\|e
			\end{equation} 
			and 
			\begin{equation}\label{ey}
				y=\left[z-\frac{1}{\lambda+1}\|w\|e\right]^-.
			\end{equation}
			From equations \eqref{epli} and \eqref{ex} follows equation \eqref{epl} and from equations \eqref{epmi} and \eqref{ey}
			follows equation \eqref{epm}, where $\lambda>0$ is given by equation \eqref{ess}, which is a combination of equations
			\eqref{ebm} and \eqref{ey}. The uniqueness of $\lambda>0$ which satisfies \eqref{ess} follows from the uniqueness of $P_L(z,w)$ and $P_M(z,w)$. \qed
	\end{enumerate} 
\end{proof}
The next remark will recover the well known formulas for projecting onto the second order cone (see for example \cite[Proposition 3.3]{MR1885570}).
\begin{remark} \label{eq:fplc}
Let $(z,w) \in {\mathbb R}\times{\mathbb R}^q$   and \(L\) be  the second order cone. Then,  letting $u:= [z- \|w\|]^+$ and $v:= [z+ \|w\|]^+$  we  conclude that Theorem~\ref{th:pelc} implies that
 \begin{equation} \label{eq:projef}
P_{L}(z,w)= \begin{cases}
  \frac{1}{2} \left( u + v ,\, \left[v -u\right] \displaystyle \frac {w}{\|w\|}\right), & w \neq 0, \\
\\
 \ \left(z^+,\,  0\right), & w= 0.
\end{cases}
\end{equation}
Indeed, for $p=1$,  the conditions in item 3 in Theorem~\eqref{th:pelc} hold if and only if $0\leq |z|< \|w\|$
and equation \eqref{ess} becomes $\lambda \|w\|= [(\lambda+1)z-\|w\|]^-$, which obviously can have only nonnegative solutions, because the 
right hand side of the equation is  nonnegative. Moreover, $\lambda=0$ cannot be a solution  because that would imply $|z|-\|w\|\ge
z-\|w\|>0$. Hence, the conditions in item 3 hold if and only if \eqref{ess} becomes $\lambda\|w\|=(\|w\|-(\lambda+1)z)$. This latter equation has 
the unique positive solution
\begin{equation}\label{eq:lambda}
 	\lambda=\frac{\|w\|-z}{\|w\|+z}.
\end{equation}
By using equation \eqref{epl} and \eqref{eq:lambda}, it is just a matter of algebraic manipulations to check that \eqref{eq:projef} holds for
this case. The cases described by items 1 and 2 can be similarly checked.
\end{remark}
\section{Numerical methods for projecting} 
In this section we  present three well known  numerical methods to find the unique zero of the   piecewise linear equation \eqref{ess}, in order to project  onto the extended second order cones. We note that    $(z,w)\in{\mathbb R}^p\times{\mathbb R}^q$ satisfies the  two 
conditions in item 3 of Theorem~\ref{th:pelc} if and only if 
\begin{equation} \label{eq: item3}
\exists ~ i_0\in \{1, \ldots, p\};  \quad  0\leq z_{i_0}^+<\|w\|,  \qquad 0\leq  \sum_{i=1}^p z_i^{-} < \|w\|. \qquad 
\end{equation}
{\it Throughout this section we will assume that $(z,w)\in{\mathbb R}^p\times{\mathbb R}^q$ satisfies \eqref{eq: item3}}.
\subsection{Semi-smooth Newton method}
In order to  study  \eqref{ess},  we consider the  piecewise linear function  $\psi$ defined by \eqref{eq:psy}. It follows from
Proposition~\ref{pr:psy} that   $\psi$ is  convex  and its unique zero,  namely   $\lambda_*>0$,  is the solution of \eqref{ess}.    The {\it semi-smooth Newton method}   for  finding the zero of  $\psi$,  with a starting point   $\lambda_{0}\in (0, +\infty)$, it  is formally  defined by 
\begin{equation} \label{eq:nmg}
\psi(\lambda_k)+ s_k\left(\lambda_{k+1}-\lambda_{k}\right)=0,  \qquad   s_k \in \partial \psi(\lambda_k), \qquad k=0,1,\ldots,
\end{equation}
where  $  s_k$ is any subgradient in  $ \partial \psi(\lambda_k)$.     
Let $N(\lambda)$ be defined by equation \eqref{def:N(x)}.
Item 1 of Proposition~\ref{pr:psy}   implies that  $-\|w\|+ \left\langle e, N(\lambda)z\right\rangle
\in \partial \psi(\lambda)$. Since $N(\lambda) [(\lambda+1)z-\|w\|e]= [(\lambda+1)z-\|w\|e]^-$,  by setting   $s_k= -\|w\| + \left\langle e,
N_kz\right\rangle$  with 
\begin{equation} \label{def:Nk}
N_k:= N(\lambda_k), 
\end{equation} 
 equation \eqref{eq:nmg} implies   
 $$
 -\lambda_k \|w\|+\langle e, N_k \left[(\lambda_k+1)z-\|w\|e\right]\rangle+ \left[ -\|w\| + \left\langle e, N_kz\right\rangle\right]\left[\lambda_{k+1}-\lambda_{k}\right]=0.
 $$
 After simplification, we get
\begin{equation} \label{eq:nm}
\left[- \|w\| +\left\langle e, N_kz\right\rangle\right] \lambda_{k+1}= - \left\langle e, N_k\left[z-\|w\|e\right]\right\rangle, \qquad k=0,1,\ldots,
\end{equation}
which  formally defines the {\it semi-smooth Newton  sequence} $\{\lambda_{k}\}$    for solving  \eqref{ess}.  
 \begin{remark}
 For $p=1$,  the conditions in \eqref{eq: item3} hold if and only if $0\leq |z|< \|w\|$. Thus, if $z\leq 0$, then  $N_k\equiv -1$
 and $\lambda_{k+1}=[\|w\|-z]/[\|w\|+z]$ for  all $k=0,1,\ldots$. Now, if $z > 0$ then letting $ 0<\lambda_0 < [\|w\|-z]/z$, we have   $N_0\equiv
 -1$ and $\lambda_{1}=[\|w\|-z]/[\|w\|+z]$. Therefore, from Remark~\ref{eq:fplc}, we conclude that the semi-smooth Newton sequence \eqref{eq:nm}
 solves equation  \eqref{ess} for $p=1$ with only one iteration.
 \end{remark}
 The proof of the next proposition is based on ideas similar to some arguments in \cite{MR3464994}.
\begin{proposition}\label{teo-finite}
 For any $\lambda_0>0$  the sequence $\{\lambda_k\}$ defined in \eqref{eq:nm} is well defined and converges after at most $2^p$ steps to the unique solution $\lambda_*>0$ of \eqref{ess}.
\end{proposition}
\begin{proof} Proposition~\ref{pr:psy} implies that $\psi$      is  convex  and  $-\|w\|+ \left\langle e, N(\lambda)z\right\rangle \in \partial \psi(\lambda)$. Thus,  we have
\begin{equation} \label{eq-chave}
 \psi(\mu )-   \psi(\lambda) -[-\|w\|+ \left\langle e, N(\lambda)z\right\rangle](\mu-\lambda) \geq 0, \qquad  \mu, \lambda \in [0, +\infty).
\end{equation}
On the other hand,  it follows from  \eqref{eq:nmg} and \eqref{def:Nk} that the sequence  $\{\lambda_k\}$  is equivalently defined as follows   
\begin{equation} \label{eq:nmg2}
\psi(\lambda_k)+ \left[- \|w\| +\left\langle e, N_kz\right\rangle\right]\left(\lambda_{k+1}-\lambda_{k}\right)=0,  \qquad   k=0,1,\ldots.
\end{equation}
By combining the above equality with the definition in   \eqref{def:Nk} and the equality in  \eqref{eq-chave}, we can   conclude that 
\begin{equation}\label{Fnegativa}
\psi(\lambda_{k+1} )\geq  \psi(\lambda_{k}) +[-\|w\| + \left\langle e, N_{k}z\right\rangle](\lambda_{k+1}-\lambda_{k})=0,  \qquad \, k=0,1,\ldots.
\end{equation}
 By letting $\mu=\lambda^{*}$ and $\lambda=\lambda_k$ in inequality \eqref{eq-chave} and by using again the definition in   \eqref{def:Nk}, we  obtain that 
\begin{equation}\label{13insolution}
0=\psi(\lambda_* )\geq   \psi(\lambda_k) + [-\|w\|+ \left\langle e, N_kz\right\rangle](\lambda_*-\lambda_k),  \qquad \, k=0,1,\ldots.
\end{equation} 
Proposition~\ref{pr:psy}  implies that  $-\|w\| +\left\langle e, N_kz\right\rangle <0$, for all $k=0,1,\ldots$. Then,  by dividing both sides of 
\eqref{13insolution} by $-\|w\| +\left\langle e, N_kz\right\rangle$ and by using \eqref{eq:nmg2}, after some algebras we obtain
\begin{equation}\label{I+}
\lambda_{k+1}= \lambda_k-[-\|w\| +\left\langle e, N_kz\right\rangle]^{-1}\psi(\lambda_k)\le \lambda_*,  \qquad \, k=0, 1, \ldots.
\end{equation} 
On the other hand, $\psi(\lambda_{k} )\geq 0 $, for all $k=0,1,\ldots$. Thus, after dividing both sides of the equality in 
\eqref{Fnegativa} by $\|w\| -\left\langle e, N_kz\right\rangle$ and some algebraic manipulations, we conclude 
\begin{equation}\label{increasing}
0<\lambda_{k}\le \lambda_k-[-\|w\| +\left\langle e, N_kz\right\rangle]^{-1}\psi(\lambda_k)= \lambda_{k+1},  \qquad \, k=0, 1, \ldots.
\end{equation} 
Hence, by combining  \eqref{I+} with \eqref{increasing}, we conclude that $ 0<\lambda_{k}\le  \lambda_{k+1}\le \lambda_*$ , for all  $k=0,1,\ldots$. Hence, $\{\lambda_{k}\}$ converges to some  $\bar{\lambda}>0$. By using again \eqref{eq:nmg2} and that  the entries of  $N_k$ are equal to $0$ or
$-1$, we have
\begin{eqnarray*}
|\psi(\bar{\lambda})|=\lim_{k\to\infty}|\psi(\lambda_k)|=\lim_{k\to\infty}|\left[- \|w\| +\left\langle e,
N_kz\right\rangle\right]\left(\lambda_{k+1}-\lambda_{k}\right)|\\\le \left[ \|w\| +\left\langle e,
|z|\right\rangle\right]\lim_{k\to\infty}|\lambda_{k+1}-\lambda_{k}|=0.
\end{eqnarray*} 
Hence, $\{\lambda_{k}\}$ converges   to  $\bar{\lambda}=\lambda_*$ the unique zero of $\psi$, which is the solution of \eqref{ess}.

Finally, we establish the finite  termination of the sequence $\{\lambda_{k}\}$ at $\lambda_{*}$,  the unique solution of \eqref{ess}. Since the
entries of  ${N}(\lambda)$ are equal to $0$ or $-1$, $N(\lambda)$  has at most $2^p$ different possible
configurations.    Then,  there exist $j, \ell \in {\mathbb N}$ with  $1\leq j<2^p$ and $1\leq \ell < 2^p$ such that $N(\lambda_j)=N(\lambda_{j+\ell})$.  Hence,  from  \eqref{eq:nm} we have 
\begin{eqnarray*}
\lambda_{j+1}=-\left[ -\|w\| +\left\langle e, N_jz\right\rangle\right]^{-1}\left\langle e, N_j\left[z-\|w\|e\right]\right\rangle\\=-\left[- \|w\| +\left\langle e, N_{j+\ell}z\right\rangle\right]^{-1}\left\langle e, N_{j+\ell}\left[z-\|w\|e\right]\right\rangle=\lambda_{j+\ell+1}.
\end{eqnarray*}
By applying this argument inductively,  $\lambda_{j+1}=\lambda_{j+\ell+1}$, $\lambda_{j+2}=\lambda_{j+\ell+2}$, $\ldots$,
$\lambda_{j+\ell}=\lambda_{j+2\ell}$, $\lambda_{j+\ell+1}=\lambda_{j+2\ell+1}=\lambda_{j+1}.$ Thus, by using \eqref{increasing} and the last
equality, we conclude that
$$
\lambda_{j+1}\leq  \lambda_{j+2} \leq \dots \leq \lambda_{j+\ell+1}\leq \lambda_{j+1}.
$$
Hence, $\lambda_{j+1}= \lambda_{j+2}$ and in view of \eqref{eq:nmg2} we conclude that $\psi(\lambda_{j+1})=0$ and $\lambda_{j+1}$ is the solution of
\eqref{ess}, i.e., $\lambda_{j+1}=\lambda_{*}$.  \qed
\end{proof}

The next proposition shows that under a further restriction on the point which is projected the convergence of the semi-smooth Newton sequence is linear.

\begin{proposition} 
	Assume that $0<\alpha<1$ and $\left\langle e, |z|\right\rangle<\alpha(1+\alpha)^{-1}\|w\|$. Then,  for any $\lambda_0>0$ , the sequence $\{\lambda_k\}$  in \eqref{eq:nm} is well defined and  converges linearly  to the unique solution $\lambda_*$ of \eqref{ess}:
\begin{equation} \label{eq:cr}
|\lambda_* -\lambda_{k+1}|\leq \alpha |\lambda_*-\lambda_k|, \qquad k=0,1,\ldots.
\end{equation}
\end{proposition}
\begin{proof}
Proposition~\ref{pr:psy} and \eqref{def:Nk} imply  $-\|w\| +\left\langle e, N_kz\right\rangle < 0$ for all $k=0,1,\ldots$, 
which implies that  the sequence $\{\lambda_k\}$ is well defined. Proposition~\ref{pr:psy} also implies that  \eqref{ess} has a zero  $\lambda_*\in  (0, +\infty)$. Hence,  by using \eqref{def:Nk},   \eqref{eq:nm} and  the definition of  $\psi$,  after some
algebra  we obtain that 
\begin{multline*}
\lambda_* -\lambda_{k+1}= \left[ -\|w\| +\left\langle e, N_kz\right\rangle\right]^{-1}  \big{[} \lambda_* \|w\|-\langle e, [(\lambda_*+1)z-\|w\|e]^-\rangle   \\
-\lambda_k \|w\|+\langle e, [(\lambda_k+1)z-\|w\|e]^-\rangle+ \left[ -\|w\| +\left\langle e, N_kz\right\rangle\right]\left[ \lambda_* -\lambda_{k}\right]\big{]}, 
\end{multline*}
for all $k=0,1,\ldots.$
On the other hand, since  $N(\lambda) [(\lambda+1)z-\|w\|e]= [(\lambda+1)z-\|w\|e]^-$, after some calculations  we have
\begin{multline*}
 \lambda_* \|w\|-\langle e, [(\lambda_*+1)z-\|w\|e]^-\rangle -  \\ \lambda_k \|w\|+ \langle e, [(\lambda_k+1)z-\|w\|e]^-\rangle+   \left[ -\|w\| +\left\langle e, N_kz\right\rangle\right]\left[ \lambda_* -\lambda_{k}\right]= \\
-\langle e,  N_*[(\lambda_*+1)z-\|w\|e]\rangle +\langle e, N_k[(\lambda_k+1)z-\|w\|e]\rangle + \left\langle e, N_kz\right\rangle \left[ \lambda_* -\lambda_{k}\right], 
\end{multline*}
for all $k=0,1,\ldots$, where $N_*:= N(\lambda_*)$.  By combining the above two equalities, we obtain that
\begin{multline*} \label{eq:taylor1}
\lambda_* -\lambda_{k+1}= \left[ -\|w\| +\left\langle e, N_kz\right\rangle\right]^{-1}  \big{[}- \langle e,  N_*[(\lambda_*+1)z-\|w\|e]\rangle +\\ \langle e, N_k[(\lambda_k+1)z-\|w\|e]\rangle + \left\langle e, N_kz\right\rangle \left[ \lambda_* -\lambda_{k}\right]  \big{]} .
\end{multline*}
Define the auxiliary piecewise linear   convex function $\zeta(\lambda) :=  \langle e,  N(\lambda)[(\lambda+1)z-\|w\|e]\rangle$. Thus, except possibly at $p$ 
points, $\zeta$ is differentiable and there holds
$$
\zeta(\lambda_*)= \zeta(\lambda_k) + \int_{0}^{1}  \left\langle e, N(\lambda_k + t(\lambda_*-\lambda_k))z\right\rangle[\lambda_*-\lambda_k]dt, 
$$
due to $\left\langle e, N(\lambda)z\right\rangle \in \partial \zeta (\lambda)$; see  \cite[Remark 4.2.5, pag. 26]{HiriartUrrutyLemarechal1993}. Hence, by simple combination of  the two latter equalities,  we have
\begin{multline*}
\lambda_* -\lambda_{k+1}=\\ -\left[- \|w\| +\left\langle e, N_kz\right\rangle\right]^{-1}  \int_{0}^{1}  \left\langle e, \left[N(\lambda_k +t(\lambda_*-\lambda_k))-N_k\right]z\right\rangle dt  [\lambda_*-\lambda_k],
\end{multline*}
for all $k=0,1,\ldots.$ Since \eqref{def:N(x)} implies that the entries of the matrix $N$ are equal to $0$ or $-1$, we obtain 
$$
	| \left\langle e, \left[N(\lambda_k + t(\lambda_*-\lambda_k))-N_k\right]z\right\rangle|\leq \sum_{j=1}^p|z_j|=\left\langle e, |z|\right\rangle.
$$
Thus, combining above equality with last inequality, we obtain  that 
$$
|\lambda_* -\lambda_{k+1}|\leq  |\|w\| -\left\langle e, N_kz\right\rangle|^{-1} \left\langle e, |z|\right\rangle |\lambda_*-\lambda_k|, \qquad k=0,1,\ldots.
$$ 
 Therefore, as we are under the assumption
$\left\langle e, |z|\right\rangle<\alpha(1+\alpha)^{-1}\|w\|$,  we have $  \left\langle e, |z|\right\rangle/[\|w\| -\left\langle e,N_kz\right\rangle] <\alpha<1$, 
\eqref{eq:cr} holds and  the sequence $\{\lambda_k\}$ converges to $\lambda_*$, which concludes the proof.\qed
\end{proof}
\subsection{Picard's method}
In this section we present a method based on Picard's iteration for solving equation \eqref{ess} under a further restriction on the point which
is projected. The statement of the result is as follows:
\begin{proposition} 
 If $\left\langle e, |z|\right\rangle<\|w\|$, then for all $\lambda_0>0$ the sequence given by the iteration
\begin{equation} \label{eq:pm}
\lambda_{k+1} =\frac{1}{\|w\|}\left\langle e, [(\lambda_k+1)z-\|w\|e]^-\right\rangle, \qquad k=1, \ldots, 
\end{equation}
converges to the unique solution of the semi-smooth equation \eqref{ess}.
\end{proposition}
\begin{proof}
It is sufficient  to prove  that $\varphi \colon [0, +\infty) \to  {\mathbb R} $ defined by 
$$
\varphi(\lambda)=\frac{1}{\|w\|}\left\langle e, [(\lambda+1)z-\|w\|e]^-\right\rangle
$$
is a contraction. Indeed, the  definition of $\varphi$ implies
\begin{align*}
|\varphi(\lambda) - \varphi(\mu)|&\leq \frac{1}{\|w\|} \sum_{i=1}\left|[(\lambda+1)z_i-\|w\|]^- - [(\mu+1)z_i-\|w\|]^-\right| \\    
                                                  &\leq \frac{1}{\|w\|} \sum_{i=1}\left|z_i (\lambda -\mu)\right| =\frac{\left\langle e, |z|\right\rangle}{\|w\|}\left|\lambda -\mu\right|, \qquad \lambda,  \mu \in  [0, +\infty).                                         
\end{align*}
Since we are under the assumption $\left\langle e, |z|\right\rangle<\|w\|$, the last inequality implies that $\varphi$ is a contraction and  the result follows. \qed \end{proof} 
\section*{Final remarks}

The {\it extended  second order cones} (ESOCs) are likely the most natural extensions of the second order cones. Also,
the complementarity problems defined on them often have nice computational properties
as remarked in the introduction. Finally, we found almost closed-form formulas
for projecting onto them. The formulas depend only on a piecewise linear
equation for a real parameter. Not so much of the ESOCs is known, nevertheless,
we stipulate that they will become an important class of cones in optimization.

For a given point in the ambient space the projection can be obtained easily in at
most $2^p$ steps, by assigning signs to the components of the second vector in the
scalar product on the right hand side of the piecewise linear equation  \eqref{ess}, solving
for $\lambda$, and if there is a solution, then checking that the solution corresponds to 
the a priori assumed signs. However, this method is computationally unviable for
larger $p$. Therefore, we developed numerical methods for solving \eqref{ess} based on the
semismooth Newton method and Picards iterations. Although the semismooth
Newton method always converges in at most  $2^p$ steps, it needs some restriction
on the point which is projected to prove that is globally linearly convergent. A
similar type of restriction is needed for Picard's method to prove that it is globally
convergent.

The complexity of our projection method is considerably lower than the complexity
of solving the reformulation of the projection problem into a second order conic
optimization problem. It is expected that there are other conic optimization problems
with respect to the extended second order cone which are easier to solve than
transforming them into second order conic optimization problems. We plan to
solve conic optimization and complementarity problems on the extended second
order cone (similarly to the second order cone in  \cite{BCFNP2017}) and to find practical examples
which can be modeled by such problems. Early studies of Lianghai Xiao (PhD
student of the second author) suggest that the extended second order cones could
be useful for portfolio selection, see \cite{MR0103768,MR2925782} and signal processing problems,  see \cite{MR3618179,MR3674429}.




\begin{thebibliography}{10}


\bibitem{AG2003}
Alizadeh, F., Goldfarb, D.: Second-order cone programming.
\newblock Math. Program. \textbf{95}(1, Ser. B), 3--51 (2003).


\bibitem{MR3464994}
Barrios, J.G., Bello~Cruz, J.Y., Ferreira, O.P., N{\'e}meth, S.Z.: A
  semi-smooth {N}ewton method for a special piecewise linear system with
  application to positively constrained convex quadratic programming.
\newblock J. Comput. Appl. Math. \textbf{301}, 91--100 (2016).


\bibitem{BCFNP2017}
Bello~Cruz, J.Y., Ferreira, O.P., N\'emeth, S., Prudente, L.F.: A semi-smooth
  {N}ewton method for projection equations and linear complementarity problems
  with respect to the second order cone.
\newblock Linear Algebra Appl. \textbf{513}, 160--181 (2017)

\bibitem{MR2179239}
Chen, J.S., Tseng, P.: An unconstrained smooth minimization reformulation of
  the second-order cone complementarity problem.
\newblock Math. Program. \textbf{104}(2-3, Ser. B), 293--327 (2005).

\bibitem{MR3618179}
Chi, C.Y., Li, W.C., Lin, C.H.: Convex optimization for signal processing and communications.
\newblock CRC Press, Boca Raton, FL (2017).
\newblock From fundamentals to applications

\bibitem{MR1885570}
Fukushima, M., Luo, Z.Q., Tseng, P.: Smoothing functions for second-order-cone
  complementarity problems.
\newblock SIAM J. Optim. \textbf{12}(2), 436--460 (2001/02).

\bibitem{MR3158056}
Gajardo, P., Seeger, A.: Equilibrium problems involving the {L}orentz cone.
\newblock J. Global Optim. \textbf{58}(2), 321--340 (2014).

\bibitem{MR3674429}
Gazi, O.: Understanding digital signal processing, \emph{Springer Topics in
  Signal Processing}, vol.~13.
\newblock Springer, Singapore (2018).

\bibitem{GowdaTao2014}
Gowda, M.S., Tao, J.: On the bilinearity rank of a proper cone and
  {L}yapunov-like transformations.
\newblock Math. Program. \textbf{147}(1-2, Ser. A), 155--170 (2014).


\bibitem{GowdaTrott2014}
Gowda, M.S., Trott, D.: On the irreducibility, {L}yapunov rank, and
  automorphisms of special {B}ishop-{P}helps cones.
\newblock J. Math. Anal. Appl. \textbf{419}(1), 172--184 (2014).


\bibitem{HiriartUrrutyLemarechal1993}
Hiriart-Urruty, J.B., Lemar{\'e}chal, C.: Convex analysis and minimization
  algorithms. {I}, \emph{Grundlehren der Mathematischen Wissenschaften
  [Fundamental Principles of Mathematical Sciences]}, vol. 305.
\newblock Springer-Verlag, Berlin (1993).
\newblock Fundamentals

\bibitem{KCY2011}
Ko, C.H., Chen, J.S., Yang, C.Y.: Recurrent neural networks for solving
  second-order cone programs.
\newblock Neurocomputing \textbf{74}, 3464--3653 (2011)

\bibitem{MR2377196}
Kong, L., Xiu, N., Han, J.: The solution set structure of monotone linear
  complementarity problems over second-order cone.
\newblock Oper. Res. Lett. \textbf{36}(1), 71--76 (2008).


\bibitem{MR1655138}
Lobo, M.S., Vandenberghe, L., Boyd, S., Lebret, H.: Applications of
  second-order cone programming.
\newblock Linear Algebra Appl. \textbf{284}(1-3), 193--228 (1998).


\bibitem{MR2568432}
Luo, G.M., An, X., Xia, J.Y.: Robust optimization with applications to game
  theory.
\newblock Appl. Anal. \textbf{88}(8), 1183--1195 (2009).

\bibitem{MR2116450}
Malik, M., Mohan, S.R.: On {$\bf Q$} and {${\bf R}_0$} properties of a
  quadratic representation in linear complementarity problems over the
  second-order cone.
\newblock Linear Algebra Appl. \textbf{397}, 85--97 (2005).

\bibitem{MR0103768}
Markowitz, H.M.: Portfolio selection: {E}fficient diversification of investments.
\newblock Cowles Foundation for Research in Economics at Yale University,
 Monograph 16. John Wiley \& Sons, Inc., New York; Chapman \& Hall, Ltd., London (1959)


\bibitem{MR0139919}
Moreau, J.J.: D\'ecomposition orthogonale d'un espace hilbertien selon deux
  c\^ones mutuellement polaires.
\newblock C. R. Acad. Sci. Paris \textbf{255}, 238--240 (1962)

\bibitem{NZ2016b}
N\'emeth, S., Zhang, G.: Positive operators of {E}xtended {L}orentz cones.
\newblock \texttt{\em arXiv:1608.07455v2}  (2016)

\bibitem{NZ20151}
N{\'e}meth, S.Z., Zhang, G.: Extended {L}orentz cones and mixed complementarity
  problems.
\newblock J. Global Optim. \textbf{62}(3), 443--457 (2015).


\bibitem{NZ2016a}
N{\'e}meth, S.Z., Zhang, G.: Extended {L}orentz cones and variational
  inequalities on cylinders.
\newblock J. Optim. Theory Appl. \textbf{168}(3), 756--768 (2016).


\bibitem{MR2522815}
Nishimura, R., Hayashi, S., Fukushima, M.: Robust {N}ash equilibria in
  {$N$}-person non-cooperative games: uniqueness and reformulation.
\newblock Pac. J. Optim. \textbf{5}(2), 237--259 (2009)

\bibitem{OrlitzkyGowda2016}
Orlitzky, M., Gowda, M.S.: An improved bound for the {L}yapunov rank of a
  proper cone.
\newblock Optim. Lett. \textbf{10}(1), 11--17 (2016).


\bibitem{RudolfNoyanPappAlizadeh2011}
Rudolf, G., Noyan, N., Papp, D., Alizadeh, F.: Bilinear optimality constraints
  for the cone of positive polynomials.
\newblock Math. Program. \textbf{129}(1, Ser. B), 5--31 (2011).


\bibitem{RS2016}
Sznajder, R.: The {L}yapunov rank of extended second order cones.
\newblock Journal of Global Optimization \textbf{66}(3), 585--593 (2016)

\bibitem{Trott2014}
Trott, D.W.: Topheavy and special {B}ishop-{P}helps cones, {L}yapunov rank, and
  related topics.
\newblock ProQuest LLC, Ann Arbor, MI (2014).
\newblock Thesis (Ph.D.)--University of Maryland, Baltimore County

\bibitem{MR2925782}
Ye, K., Parpas, P., Rustem, B.: Robust portfolio optimization: a conic programming approach.
\newblock Comput. Optim. Appl. \textbf{52}(2), 463--481 (2012).


\bibitem{MR2925039}
Yonekura, K., Kanno, Y.: Second-order cone programming with warm start for
  elastoplastic analysis with von {M}ises yield criterion.
\newblock Optim. Eng. \textbf{13}(2), 181--218 (2012).


\bibitem{MR3010551}
Zhang, L.L., Li, J.Y., Zhang, H.W., Pan, S.H.: A second order cone
  complementarity approach for the numerical solution of elastoplasticity
  problems.
\newblock Comput. Mech. \textbf{51}(1), 1--18 (2013).

\end{thebibliography}

\end{document}